\documentclass[12pt,amscd,amssymb]{amsart}																																					

\usepackage[centertags]{amsmath}
\usepackage{amsmath}
\usepackage{mathrsfs}
\usepackage{latexsym}
\usepackage{amsthm}
\usepackage{amsfonts}
\usepackage{amssymb}
\usepackage{stmaryrd}

\usepackage{xspace}
\usepackage{verbatim}
\usepackage{enumerate}
\usepackage{indentfirst}

\theoremstyle{plain}

\DeclareMathOperator{\sll}{SL}
\DeclareMathOperator{\gl}{GL}

\newtheorem{theorem}{Theorem}[section]

\newtheorem{fact}[theorem]{Fact}
\newtheorem*{theorem*}{Theorem}
\newtheorem*{maintheorem*}{Main Theorem}
\newtheorem*{lemma*}{Lemma}

\newtheorem*{intconj*}{Intermediate Conjecture}
\newtheorem{prop}[theorem]{Proposition}
\newtheorem*{prop*}{Proposition}
\newtheorem*{conj*}{Principal Conjecture}
\newtheorem*{thm3.2*}{Theorem 3.2}
\newtheorem*{thm3.3*}{Theorem 3.3}
\newtheorem*{prop2.11*}{Proposition 2.11}
\theoremstyle{definition}
\newtheorem{example}[theorem]{Example}

\newtheorem{remark}[theorem]{Remark}
\newtheorem*{remark*}{Remark}

\newtheorem*{claim*}{Claim}

\newcommand{\los}{\L o\'{s}'s Theorem}
\begin{document}

\title[On Sylow theory for  linear pseudofinite groups]{On Sylow theory for linear  pseudofinite groups}
\author[P. U\u{G}URLU]{P\i nar U\u{g}urlu Kowalski} 
\address{Istanbul Bilgi  University \\
Yahya K\"{o}pr\"{u}s\"{u} Sok. No: 1
Dolapdere \\34440\\
Istanbul\\
Turkey}
 \email{pinar.ugurlu@bilgi.edu.tr}

\maketitle

\begin{abstract}
We prove the conjugacy of Sylow $p$-subgroups of linear pseudofinite  groups under the assumption of the existence of a finite Sylow $p$-subgroup.  We also give an example of a linear pseudofinite  group with non-conjugate Sylow $2$-subgroups.
\end{abstract}

\section*{Introduction}
Since pseudofinite groups are models of the first order theory of finite groups, it is natural to try to generalise some results related to finite groups to the pseudofinite case. One of the nice properties shared by all finite groups is the conjugacy of Sylow $p$-subgroups. In this work, we show that this is too good to be true for pseudofinite groups, even for linear ones, by providing an example of a linear pseudofinite group with non-conjugate Sylow $2$-subgroups (see Example~\ref{ex2}). However, under  extra assumptions, one can obtain some conjugacy results. In \cite{kow}, the author proved the conjugacy of Sylow $2$-subgroups of pseudofinite $\mathfrak{M}_c$-groups (groups satisfying descending chain condition on centralizers) provided that there is a finite Sylow $2$-subgroup. Since linear groups are $\mathfrak{M}_c$-groups, this results gives the desired conjugacy of Sylow $2$-subgroups for linear pseudofinite  groups once there is a finite Sylow $2$-subgroup. In this paper, we generalise this conditional conjugacy result  to Sylow $p$-subgroups (for any prime $p$) in the case of linear  pseudofinite  groups. 

The main result of this paper is stated below.

\begin{thm3.2*}
If  a linear  pseudofinite  group $G$  has a finite Sylow $p$-subgroup, then all Sylow $p$-subgroups of $G$ are conjugate and hence finite.
\end{thm3.2*}

The structure of this paper is as follows.

In the first section, we recall some basic notions from group theory, we briefly introduce pseudofinite groups  and  we fix our terminology and notation.

In the second section, we summarize some known (non-)conjugacy results for Sylow $p$-subgroups of some particular  groups.

In the last section, we state and prove the main theorem (Theorem 3.2) and provide some examples of linear pseudofinite groups with conjugate and non-conjugate Sylow $2$-subgroups. We also comment on a question stated by Wagner in \cite{wagner2}.

\section{Preliminaries}
In this section, we recall the definitions of some basic notions from group theory and  we also recall the notion of pseudofinite groups very briefly without going into details. By doing so, we will fix our terminology and notation. Note that we assume that the reader is familiar with the basic notions in model theory. We refer the reader to the books \cite{bells} and \cite{changk} for detailed information about the ultraproduct construction (which is essential to understand pseudofinite groups) and to \cite{wilson} and \cite{ugurlu} for more information about pseudofinite groups.

 A  \emph{linear group} is a group which is isomorphic to a subgroup of $\gl_n(F)$  where $\gl_n(F)$ denotes the general linear group over a field $F$. It is easy to observe that finite groups are linear.  A group $G$ is called a \emph{$p$-group} for a prime number $p$, if each element of $G$ is a \emph{$p$-element}, that is, each element  has order $p^n$ for a natural number  $n$. More generally, $G$ is called  \emph{periodic} if every element of it has finite order. An example of an infinite $p$-group is the Pr\"{u}fer $p$-group: $$C_{p^{\infty}}= \bigcup_{n=1}^{\infty} C_{p^n},$$ where $C_k$ denotes the cyclic group of order $k$ for $k>0$. A \emph{Sylow $p$-subgroup} of a group $G$ is a  $p$-subgroup of  $G$ which is maximal with respect to inclusion. A group $G$ is called \emph{locally finite} if every finitely generated subgroup of $G$ is finite. Note that locally finite groups are necessarily periodic, but, the converse is not necessarily true in general. However, for linear groups, these two notions coincide by the following result:


 \begin{fact} [Schur \cite{schur1911ueber}] \label{lf}
Periodic linear groups are locally finite.
\end{fact}




\emph{Pseudofinite groups} are defined as infinite models of the common theory of finite groups.  It can be shown that any pseudofinite group is elementarily equivalent
to a non-principal ultraproduct of finite groups, by a suitable choice of an ultrafilter (see \cite{wilson}).  The importance of the ultraproduct construction is expressed by  \los ~ \cite{los1} which states that a first order formula is satisfied in the ultraproduct if and only if it is satisfied in the structures indexed by a set belonging to the ultrafilter.  In particular, the first order properties of the ultraproduct  are determined by the first order properties of the structures in the ultraproduct together with the choice of an ultrafilter.  Throughout the text,  if $G$ is a pseudofinite group then we write $G \equiv \prod_{i\in I} G_i/\mathcal{U}$ which means that $G$ is elementarily equivalent to the ultraproduct of the finite groups $G_i, \ i\in I$,  over a non-principal ultrafilter $\mathcal{U}$ on the index set $I$. Moreover, if a first order property holds in the finite groups indexed by a set belonging to the ultrafilter, then we say that this property holds in \emph{almost all} of the finite groups in the ultraproduct.


Since we talk about linear pseudofinite groups in this paper, it is worth to mention that unlike in the case of finite groups, there are non-linear pseudofinite groups. For an example we refer the reader to  \cite{kow}.

\section{On  (non-)conjugacy results for Sylow $p$-subgroups}
In this section, we summarize some known results about (non-)conjugacy of Sylow $p$-subgroups for some important classes of groups. We do not claim that we  include all the related results from the literature. We refer the reader to \cite{algebra} for a more comprehensive survey.

\subsection{Conjugacy results}

First we list some facts about conjugacy of Sylow $p$-subgroups in linear groups.

\begin{fact} \label{conjugacy} Sylow $p$-subgroups are conjugate in
\begin{enumerate}
   \item  periodic linear groups (Platonov \cite{platonov}),
      \item linear algebraic  groups over  algebraically closed fields (Platonov \cite{platonov}),
\item $\gl_n(K)$ except when $p=2$, char$(K)\neq 2$ and the following conditions are satisfied (Vol'vachev \cite{volvachev2}):
\begin{itemize}
  \item[${\bf (V_1)}$] $-1$ can be written as a sum of two squares in $K$,
  \item[${\bf (V_2)}$] $K$ has no element of multiplicative order $4$,
  \item[${\bf (V_3)}$]  Every $2$-element $a+bi$ in $K(i)$, where $i^2=-1$, satisfies $(a+bi)(a-bi)=1$.
\end{itemize}
\end{enumerate}
\end{fact}


Wagner  proved conjugacy results for a wider class of groups than linear groups (namely, substable groups) but with some extra assumptions. His result is stated below and for details we refer the reader to the book \cite{wagner2}.
\begin{fact}[Theorem 1.5.4, Wagner \cite{wagner2}] \label{wagnerconj}
   Let $G$  be a substable group. All Sylow $p$-subgroups of $G$ are conjugate if one of the following conditions hold:
    \begin{enumerate}
      \item $p=2$ and $G$ is ABD.
      \item $G$ is soluble by finite and ABD.
      \item $p=2$ and $G$ is periodic.
      \item Any two $p$-elements of $G$ generate a finite subgroup.
    \end{enumerate} 
\end{fact}
Note that a group $G$ has the \emph{ABD property} if every relatively definable
abelian section  decomposes as a sum of a divisible group  and a group of
bounded exponent (see Definition 1.5.1 in \cite{wagner2}).

In the case of finite Morley rank groups only the following conjugacy result is known.

\begin{fact}[Alt\i nel, Borovik and Cherlin   \cite{abc}]
Sylow $2$-subgroups of a group of finite Morley rank are conjugate.
\end{fact}

\subsection{Non-conjugacy results}

Below, we list some examples of families of groups with non-conjugate Sylow $p$-subgroups.

\begin{itemize}
  \item There are linear algebraic groups with  non-conjugate Sylow $p$-subgroups.

\noindent
  $\sll_2(\mathbb{Q})$ is an example given by Platonov in \cite{platonov}. Note that this example is analyzed in detail in the paper \cite{kow} of the author. 

\item There are locally finite groups with non-conjugate Sylow $p$-subgroups.

\noindent Direct product of Sym$(3)$'s is an example given by Dixon (Example 2.2.6 in \cite{dixon1994sylow}). This group is a metabelian locally finite group with non-conjugate Sylow $2$-subgroups.
  \item There are linear pseudofinite groups with non-conjugate Sylow $2$-subgroups (see Example~\ref{ex2}).
\end{itemize}

\section{Some results and examples in the case of pseudofinite groups}
 In this section, we prove the main result of this paper, namely, if a linear pseudofinite group has a finite Sylow $p$-subgroup, then all Sylow $p$-subgroups are conjugate.  We will also construct  examples of linear pseudofinite groups with conjugate and non-conjugate Sylow $2$-subgroups.

\subsection{Existence of finite Sylow $p$-subgroups}
Let  $(G_i)_{i\in I}$ be a family of finite groups. The first naive candidate for a Sylow $p$-subgroup of $\prod_{i\in I} G_i/\mathcal{U}$ is the ultraproduct of Sylow $p$-subgroups of the finite groups $G_i$. However, this is generally not the case. Such an ultraproduct is usually not even periodic.  For example, if we  take a non-principle ultraproduct of the cyclic groups $(C_{p^i})_{i\in \mathbb{N}}$ then clearly the ultraproduct of Sylow $p$-subgroups of $C_{p^i}$'s is not even a $p$-group; it has elements of infinite order. Actually, it is not difficult to observe that the Sylow $p$-subgroup of $\prod_{i\in \mathbb{N}}  C_{p^i}/ \mathcal{U}$ is isomorphic to the Pr\"{u}fer $p$-group  $C_{p^{\infty}}$ since clearly  $C_{p^{\infty}}$ embeds in the ultraproduct and the ultraproduct has unique elements of order $p^i$ for each $i$. So, in general the structure and behaviour of the Sylow $p$-subgroups can get complicated in pseudofinite groups. However, as  the following results suggest, existence of a finite Sylow $p$-subgroup guarantees the conjugacy of all Sylow $p$-subgroups.

\begin{prop}\label{bounded1}
Let $(G_i)_{i\in I}$ be a family of finite groups and   $G \equiv \prod_{i\in I} G_i / \mathcal{U}$ be a pseudofinite group. If $G$ has a Sylow $p$-subgroup of order $p^k$ for some prime number $p$ and positive integer $k$,  then the following holds.
\begin{enumerate}
  \item Sylow $p$-subgroups of  $G_i$ have order $p^k$ for almost all $i\in I$.
  \item There are no finite subgroups of $G$ of order $p^l$ for any $l > k$.
  \item All finite Sylow $p$-subgroups of $G$ have order $p^k$ and they are conjugate.
\end{enumerate}
\end{prop}

\begin{proof} $(a)$ It is easy to observe that the following statement is first order:
\vspace{0.2cm}

``There is a subgroup of order $p^k$ which is not contained in a subgroup of order $p^{k+1}.$"
\vspace{0.2cm}

\noindent Since this first order sentence holds in $G$ by the assumption,    almost all of the groups $G_i$ have a subgroup of order $p^k$ which is not contained in a subgroup of order $p^{k+1}$  by \los. Therefore,    the order of Sylow $p$-subgroups of almost all $G_i$'s is $p^k$ by Sylow's first theorem.

\vspace{0.2cm}

\noindent $(b)$ Assume on the contrary that $G$ has a subgroup of order $p^l$ for some $l > k$. Then again by \los \ almost all $G_i$ in the ultraproduct has subgroups of order $p^l$. This is not possible since their Sylow $p$-subgroups have order $p^k$ by (a) and $p^l > p^k$.

\vspace{0.2cm}

\noindent $(c)$  Let $P$ be a finite Sylow $p$-subgroup of $G$.  Clearly, $|P| \leqslant p^k$ by (b). If $|P|= p^l < p^k$ then this means that $P$ is not contained in a subgroup of order $p^k$ (since it is maximal finite $p$-group) and this would yield  subgroups of order $p^l$ for almost all $G_i$ in the ultraproduct which are not contained in  groups of order $p^k$. This contradicts to the fact that  Sylow $p$-subgroups of $G_i$'s have order $p^k$. For the conjugacy,  note that the following statement is also first order:

\vspace{0.2cm}

``Any two subgroups of order $p^k$ are conjugate."

\vspace{0.2cm}

\noindent Since the Sylow $p$-subgroups of almost all $G_i$ are of order $p^k$, the first order sentence above holds in the ultraproduct and hence in $G$ (by \los).
\end{proof}

\begin{theorem}\label{bounded2}
Let   $G$ be a linear pseudofinite group. If  $G$ has a finite  Sylow $p$-subgroup then all Sylow $p$-subgroups of $G$  are conjugate.
\end{theorem}
\begin{proof}
Suppose $G$ has a Sylow $p$-subgroup of order $p^k$. Thanks to  Proposition~\ref{bounded1}, it is enough to  show that $G$ has no infinite Sylow $p$-subgroups.   Assume on the contrary that $G$ has an infinite  Sylow $p$-subgroup, say $P$.   Let $x_1, x_2, \ldots, x_{p^k}, x_{p^k+1}$ be distinct elements of $P$ and define  $S:=\langle x_1, x_2, \ldots, x_{p^k}, x_{p^k+1} \rangle$. Since $G$ is a linear group, $P$ is locally finite (see Fact~\ref{lf}) and so  $S$ is a finite $p$-subgroup of $G$ whose order is strictly bigger than $p^k$. This is not possible by Proposition~\ref{bounded1}(b).     As a result, all Sylow $p$-subgroups of $G$ are finite and conjugacy follows by Proposition~\ref{bounded1}(c).
\end{proof}

\begin{remark}\label{remarkbound}  By Proposition~\ref{bounded1} and the proof of Theorem~\ref{bounded2}, we get the following results.
\begin{enumerate}
  \item There can not be both finite and infinite Sylow $p$-subgroups in a linear pseudofinite group. We do not know whether this holds generally for pseudofinite groups or not.
  \item Let $G \equiv \prod_{i\in I} G_i / \mathcal{U}$ be  a linear pseudofinite group. Then $G$ has a finite Sylow $p$-subgroup if and only if there is a finite global bound on the order of Sylow $p$-subgroups of almost all $G_i$.
\end{enumerate}

\end{remark}

\begin{remark}
The assumption of linearity in Theorem~\ref{bounded2} can be weaken to the $\mathfrak{M}_c$-property for $p=2$ since Sylow $2$-subgroups of $\mathfrak{M}_c$-groups are known to be locally finite (see Corollary 2.4 in \cite{wagner1}). Therefore,  Theorem~\ref{bounded2} gives an alternative proof of the result of the author in \cite{kow}.
\end{remark}

\subsection{(Non-)Conjugacy examples in linear pseudofinite groups}

In this section, we present two examples of linear pseudofinite groups, namely, $\gl_2(F)$ over different pseudofinite fields $F$, such that the  Sylow $2$-subgroups are conjugate in the first example while conjugacy of Sylow $2$-subgroups fails in the second one.

\begin{example} \label{ex1} Let $I=\{p \ \text{prime} ~|~ p \equiv 3 \ (\text{mod}\ 8)\}$ and $F = \prod_{p\in I} \mathbb{F}_{p} / \mathcal{U}$ where $\mathcal{U}$ is a non-principal ultrafilter on $I$. Since there are infinitely many primes in $I$  (by Dirichlet's theorem on arithmetic progression),   $F$ is a pseudofinite field of characteristic $0$. Let us consider the general linear group $\gl_2(F)$.    Then we have  $\gl_2(F) \equiv \prod_{p\in I} \gl_2(\mathbb{F}_p) / \mathcal{U}$ and  therefore, $\gl_2(F)$ is a linear pseudofinite  group.

 It is not difficult to observe that the group $\gl_2(F)$, where $F$ is as above,  has finite Sylow $2$-subgroups. To see this first note that  the order of $\gl_2(\mathbb{F}_p)$ is $(p^2-1)(p^2-p)$ and hence, the powers of $2$ dividing $p-1$ and $p+1$ determine the order of the Sylow $p$-subgroup of  $\gl_2(\mathbb{F}_p)$. Since $p \equiv 3 \ (\text{mod}\ 8)$, clearly $4 \nmid p-1$ while $2 \mid p-1$. Because of the same reason $8 \nmid p+1$ while $4 \mid p+1$. Therefore, the order of the Sylow $2$-subgroups of all the groups $\gl_2(\mathbb{F}_p)$ in the ultraproduct is $16$.  Since  this finiteness property  is transferred to the ultraproduct (see Remark~\ref{remarkbound}(b)), Sylow $2$-subgroups of the group  $\gl_2(F)$ are finite and conjugate.
 \end{example}
\begin{remark}
We could also  observe the conjugacy of Sylow $2$-subgroups of the pseudofinite group in Example~\ref{ex1} by referring to Volvachev conditions. Note that   ${\bf (V_1)}$ is a well-known property of finite fields. Since it is a first order property, ${\bf (V_1)}$ is  satisfied by any pseudofinite field. Moreover, by the  construction and \los, $F$ satisfies ${\bf (V_2)}$ as well. However,  ${\bf (V_3)}$ is not satisfied by $F$.  To see this consider  $\gl_2(\mathbb{F}_p)$ where $p  \in I$. Clearly we have   $$|\mathbb{F}_p(i)^*|=p^2-1=(p-1)(p+1)$$ where $i$ is an element of multiplicative order $4$. Since $p \equiv 3 \ (\text{mod}\ 8)$, as we have mentioned above, $8$ divides the order of the cyclic group $\mathbb{F}_p(i)^*$. So, there exist $a, b \in \mathbb{F}_p$ such that $a+bi$ has multiplicative order $8$. Now let us see why  we should have $(a+bi)(a-bi)=-1$. Firstly,  $(a+bi)(a-bi)=\pm 1$ since this product is a $2$-element in $\mathbb{F}_p^*$ and  $\mathbb{F}_p^*$ has no element of order $4$.  Now consider the following two automorphisms of $\mathbb{F}_p(i)$: $$ a+bi \mapsto a-bi, \ \ \ \ \ \ \ \  \ \ a+bi \mapsto (a+bi)^p.$$ Clearly, these two automorphisms are non-trivial elements of the Galois group Gal$(\mathbb{F}_p(i)/\mathbb{F}_p)$. Since this Galois group has order $2$, it has only one non-trivial element, that is  we have $a-bi=(a+bi)^p$ which gives:
$$(a+bi)(a-bi)=(a+bi) (a+bi)^p= (a+bi)^{8k+4}=(a+bi)^4=-1$$ since $p=8k+3$ for some integer $k$ and $a+bi$ has multiplicative order $8$. As a result we found elements $a,b$ in $\mathbb{F}_p$ such that $a+bi$ has multiplicative  order $8$ and $(a+bi)(a-bi)=-1$. Since this is true for any field $\mathbb{F}_p$ in the ultraproduct,  there are such  elements in $F$ by \los. As a result we can conclude that Sylow $2$-subgroups of the group in our example are conjugate.
\end{remark}

\begin{remark} Any field of positive characteristic satisfying condition ${\bf (V_2)}$  can not satisfy condition ${\bf (V_3)}$ (see \cite{volvachev2}). In particular no finite field satisfies ${\bf (V_2)}$ and ${\bf (V_3)}$ at the same time. Volvachev's argument about positive characteristic fields works well for pseudofinite fields of characteristic $0$, provided that there is a
 bound on the order of $2$-elements of $F(i)^*$. This is the case for the pseudofinite field $F$ in Example~\ref{ex1}.
\end{remark}

In the next example we observe that there are  linear pseudofinite groups in which Sylow $2$-subgroups are not conjugate. This happens when there is no bound on the  order of $2$-elements of $F(i)^*$ (that is when $C_{2^{\infty}}$ embeds in $F(i)^*$).

\begin{example}\label{ex2}
 We will construct a pseudofinite field $F$ satisfying the Volvachev conditions ${\bf (V_2)}$ and ${\bf (V_3)}$ (note that ${\bf (V_1)}$ is automatic for pseudofinite fields). Let $$I=\{p \ \text{prime} ~|~ p \equiv 3 \ (\text{mod}\ 4)\}.$$  For each positive integer $N$, we can find a prime number $p \in I$ such that $2^N$ divides $p+1$ by Dirichlet's theorem on arithmetic progressions (actually there are infinitely many such primes for each $N$). Let us denote this infinite subset of $I$ by $J$ and define  $F=\prod_{p\in J} \mathbb{F}_{p} / \mathcal{U}$ where $\mathcal{U}$ is a non-principal ultrafilter on $J$. Clearly ${\bf (V_2)}$ is satisfied in $F$ since each field in the ultraproduct has this first order property. Moreover, by our construction, $2^N$   divides $|\mathbb{F}_p(i)^*|=(p-1)(p+1)$ for almost all $p\in J$ for each positive integer $N$. As a result  $C_{2^\infty}$ embeds in $F(i)^*$.  Now assume that ${\bf (V_3)}$ is not satisfied, that is, $(a+bi)(a-bi)=-1$ for some $2$-element $a+bi$ in $F(i)^*$. Since $C_{2^\infty}$ embeds in $F(i)^*$,  there is a $2$-element $c+di$ such that $(c+di)^2=a+bi$. But then we have $(c+di)^2 (c-di)^2 = -1$, that is, $(c^2+d^2)^2=-1$. But this contradicts to our assumption that $F$ has no element of order $4$. Therefore the pseudofinite linear group $\gl_2(F)$, where $F$ is constructed as above, has non-conjugate Sylow $2$-subgroups by Volvachev's theorem.
\end{example}

\begin{remark}
It is not difficult to observe that the group $\gl_2(F)$ constructed in   Example~\ref{ex2} has infinite Sylow $2$-subgroups. As mentioned in Remark~\ref{remarkbound}(b) it is enough to observe that   there is no bound on the orders of the Sylow $2$-subgroups of the groups $\gl_2(\mathbb{F}_p)$ in the ultraproduct. This follows easily since the powers of $2$-divisors of $p+1$ grow in the ultraproduct by construction.

\end{remark}
We would like to finish this paper with a comment about an open problem which is stated by Wagner (Remark 1.5.1 in \cite{wagner2}).

In \cite{wagner2}, Wagner points out that ``The question of how conjugacy of the
maximal $p$-subgroups behaves under elementary equivalence is still
open". The connection of our results to this question may be explained as follows. For convenience, let $\mathfrak{S}$ denote the class of groups in
which Sylow $p$-subgroups are conjugate for each prime $p$.
\begin{enumerate}
  \item Our Example~\ref{ex2} shows that the class $\mathfrak{S}$ is not closed
under ultraproducts.
  \item  It is well-known that a class is elementary (that is: axiomatizable
by a first-order theory) if and only if it is closed under
ultraproducts and elementary equivalence. Therefore, the class
$\mathfrak{S}$ is not  elementary by (a).
  \item  The observations mentioned in (a) and (b) above suggest that
 the answer to the question stated by Wagner  is negative. However, we do not have an example of two
elementarily equivalent groups such that one belongs to $\mathfrak{S}$ while the other does not.
\end{enumerate}

\noindent
{\it Acknowledgements}. \  I would like to thank Zo\'e Chatzidakis for her valuable comments related to Example~\ref{ex2}. This work has been completed during my sabbatical leave in  2022--2023 Academic Year which I have been spending at Wroc{\l}aw University.

\bibliographystyle{plain}
\bibliography{refref}
\end{document}